\documentclass[12pt,a4paper]{amsart}
\usepackage[margin=2.5cm]{geometry}
\usepackage{amssymb}
\usepackage{graphicx} 
\usepackage[mathscr]{euscript}
\usepackage{enumerate}
\usepackage{xspace}
\usepackage{color}
\usepackage{enumerate}
\usepackage{enumitem}
\usepackage{amsthm}
\usepackage{dsfont}
\usepackage{bbm}
\usepackage[colorlinks=true, linkcolor = blue, citecolor = blue]{hyperref}
\usepackage[numbers]{natbib}
\usepackage[backgroundcolor=white,bordercolor=red]{todonotes}
\DeclareMathAlphabet{\mathpzc}{OT1}{pzc}{m}{it}
\usepackage[nameinlink]{cleveref}
\begin{document}

\theoremstyle{plain}

\newtheorem{theorem}{Theorem}[section]
\newtheorem{lemma}[theorem]{Lemma}
\newtheorem{proposition}[theorem]{Proposition}
\newtheorem{corollary}[theorem]{Corollary}
\newtheorem{definition}[theorem]{Definition}
\newtheorem{Ass}[theorem]{Assumption}
\theoremstyle{definition}
\newtheorem{remark}[theorem]{Remark}
\newtheorem{SA}[theorem]{Standing Assumption}
\newtheorem*{observation}{Observations}
\newtheorem*{RL}{Comments on Related Literature}

\renewcommand{\chapterautorefname}{Chapter} 
\renewcommand{\sectionautorefname}{Section} 

\crefname{lemma}{lemma}{lemmas}
\Crefname{lemma}{Lemma}{Lemmata}
\crefname{corollary}{corollary}{corollaries}
\Crefname{corollary}{Corollary}{Corollaries}

\newcommand{\m}{\mathfrak{m}}
\newcommand{\n}{\mathfrak{n}}
\newcommand{\q}{\mathfrak{q}}
\def\stackrelboth#1#2#3{\mathrel{\mathop{#2}\limits^{#1}_{#3}}}
\newcommand{\M}{\mathcal{M}_1}

\renewcommand{\theequation}{\thesection.\arabic{equation}}
\numberwithin{equation}{section}

\newcommand{\1}{\mathds{1}}
\renewcommand{\epsilon}{\varepsilon}
\newcommand{\X}{\mathsf{X}}
\newcommand{\Z}{\mathsf{Z}}
\newcommand{\B}{\mathsf{B}}

\title[On the Relation of Diffusions and their Speed measures]{On the Relation of One-Dimensional Diffusions on Natural Scale and their Speed measures} 
\author[D. Criens]{David Criens}
\address{D. Criens - Albert-Ludwigs University of Freiburg, Ernst-Zermelo-Str. 1, 79104 Freiburg, Germany}
\email{david.criens@stochastik.uni-freiburg.de}

\keywords{diffusion, speed measure, homeomorphism, convergence of diffusions, sufficient and necessary conditions, limit theorem, vague convergence, weak convergence, Feller-Dynkin property, It\^o diffusion.\vspace{1ex}}


\thanks{Financial support from the DFG project No. SCHM 2160/15-1 is gratefully acknowledged. \\}

\date{\today}
\maketitle

\frenchspacing
\pagestyle{myheadings}

\begin{abstract}
It is well-known that the law of a one-dimensional diffusion on natural scale is fully characterized by its speed measure. 
C.~Stone proved a continuous dependence of diffusions on their speed measures. 
In this paper we establish the converse direction, i.e. we prove a continuous dependence of the speed measures on their diffusions. 
Furthermore, we take a topological point of view on the relation. 
More precisely, for suitable topologies, we establish a homeomorphic relation between the set of regular diffusions on natural scale without absorbing boundaries and the set of locally finite speed measures. 
\end{abstract}

\section{On the Relation of Diffusions and their Speed Measures}
\subsection{Introduction}
It is well-known (see e.g. \cite{breiman1968probability,itokean74}) that the law of a one-dimensional regular continuous strong Markov process on natural scale (called \emph{diffusion} in this short section) is fully characterized by its speed measure. 
Among other things, Stone \cite{stone} proved that diffusions depend continuously on their speed measures and Brooks and Chacon~\cite{brooks} established the converse direction for real-valued diffusions, i.e. they proved a continuous dependence of the speed measures on the diffusions. 

The main contribution of this paper is a converse to Stone's theorem for general diffusions. 
The real-valued and the general case distinguish in two important points: For real-valued diffusions there is no issue with the boundary behavior and the corresponding speed measures are locally finite, which in particular means they can be endowed with the vague topology.
To treat the general case we use a new method of proof, which is quite different to those of Brooks and Chacon. Below we comment in more detail on the methods and compare them to each other.

As a second contribution, we consider the relation of diffusions and their speed measures from a topological perspective. Namely, for suitable topologies, we deduce a homeomorphic relation between the set of regular diffusions on natural scale without absorbing boundaries and the set of locally finite speed measures. As an application of the homeomorphic relation we discuss properties of certain subsets of the set of diffusions without absorbing boundaries, namely those with the Feller--Dynkin property and It\^o diffusions with open state space. More precisely, we show that both of these subsets are dense Borel sets which are neither closed nor open.

The remainder of this paper is structured as follows. In Section \ref{sec: basic setting} we introduce our setting and recall the most important terminologies. Thereafter, in Section \ref{sec: main} we present our results and we comment on related literature. Finally, in Section \ref{sec: pf} we present the proof of our main theorem, i.e. the converse to Stone's theorem.

\subsection{Setting} \label{sec: basic setting}
We work with the canonical setting for diffusions as introduced in \cite[Section~V.25]{RW2}. A quite complete treatment of the theory is given in the monograph of It\^o and McKean \cite{itokean74}. Shorter introductions can also be found in the monographs \cite{breiman1968probability,RY}.
Let \(J \subset \mathbb{R}\) be a finite or infinite, closed, open or half-open interval and define \(\Omega\) to be the space of continuous functions \(\mathbb{R}_+ \to J\) endowed with the local uniform topology. The coordinate process on \(\Omega\) is denoted by \(\X\), i.e. \(\X_t (\omega) = \omega(t)\) for \(t \in \mathbb{R}_+\) and \(\omega \in \Omega\). 
The Borel \(\sigma\)-field on \(\Omega\) is given by \(\mathcal{F} \triangleq \sigma (\X_s, s \geq 0)\).
For any time \(t \in \mathbb{R}_+\) we also set \(\mathcal{F}_t \triangleq \sigma (\X_s, s \leq t)\) and we define the shift operator \(\theta_t \colon \Omega \to \Omega\) by \((\theta_t \omega) (s) = \omega (t + s)\) for \(s, t \in \mathbb{R}_+\). Let \(\M\) be the set of probability measures on \((\Omega, \mathcal{F})\) endowed with the weak topology.

We call \((J \ni x \mapsto P_x \in \M)\) a \emph{(canonical) diffusion}, if \(x \mapsto P_x(A)\) is measurable for all \(A \in \mathcal{F}\),\footnote{this is equivalent to saying that \((x \mapsto P_x)\) is \(\mathcal{B}(J)/\mathcal{B}(\M)\) measurable, see \cite[Theorem 19.7]{aliprantis}} \(P_x(\X_0 = x) = 1\) for all \(x \in J\), and for any \((\mathcal{F}_{t+})_{t \geq 0}\)-stopping time \(\tau\) and any \(x \in J\), \(P_{\X_\tau}\) is the regular conditional \(P_x\)-distribution of \(\theta_{\tau} \X\) on \(\{\tau < \infty\}\). The final part is the strong Markov property. 
A diffusion \((x \mapsto P_x)\) is called \emph{regular} if for all \(x \in J^\circ\) and~\(y \in J\)
\begin{align}\label{eq: regular}
P_x(\gamma_y < \infty) > 0,
\end{align}
where \(\gamma_y \triangleq \inf (s \geq 0
 \colon \X_s = y)\), and it is called \emph{completely regular}\footnote{this terminology is new in the sense that it does not appear in \cite{breiman1968probability, itokean74,RY,RW2}} if \eqref{eq: regular} holds for all \(x, y \in J\).
 Clearly, regularity and complete regularity are equivalent for open~\(J\). In case \(J\) is closed or half-open, complete regularity means that closed boundaries are reflecting (i.e. neither exit nor absorbing in the language from \cite[Section 16.7]{breiman1968probability}). 
We say that a regular diffusion \((x \mapsto P_x)\) is on \emph{natural scale} if for all \(a, b, x \in J\) with \(a < x < b\) we have 
\[
P_x (\gamma_b < \gamma_a) = \frac{x - a}{b - a}. 
\]
Any regular diffusion can be brought to natural scale via a homeomorphic space transformation (\cite[Proposition~16.34]{breiman1968probability}). 
Let \((x \mapsto P_x)\) be a regular diffusion on natural scale.
According to \cite[Theorem 16.36]{breiman1968probability}, there exists a unique locally finite measure \(\m\) on \((J^\circ, \mathcal{B}(J^\circ))\) such that for any \(a < b\) with \([a, b] \subset J^\circ\) we have 
\[
E_x \big[ \gamma_a \wedge \gamma_b\big] = \int G_{(a, b)} (x, y) \m (dy), \quad x \in [a, b], 
\]
where \(G_{(a, b)}\) is the Green function as defined in \cite[Eq. 16.35]{breiman1968probability}. Furthermore, by \cite[Theorem~16.47]{breiman1968probability}, if the left boundary point \(l\) is in~\(J\), then \(\m(\{l\})\) can be defined such that for any~\(b \in J^\circ\)
\[
E_x \big[ \gamma_b \big] = \int G_{[l, b)} (x, y) \m (dy), \quad x \in [l, b], 
\]
where \(G_{[l, b)}\) is the symmetrized Green function as defined in \cite[Eq. 16.46]{breiman1968probability}. A similar statement holds for right boundary points which are in \(J\). In this manner we get a measure \(\m\) on \((J, \mathcal{B}(J))\) which is called the \emph{speed measure} associated to the regular diffusion \((x \mapsto P_x)\). The speed measure is locally finite if and only if the corresponding diffusion is completely regular.
Within the class of regular diffusions the speed measure determines a diffusion uniquely (\cite[Corollary 16.73]{breiman1968probability}).

\subsection{Main results} \label{sec: main}
In the following let \(l\) be the left boundary point of \(J\) and let \(r\) be the right boundary point.
Take a sequence \(\m^0, \m^1, \m^2, \dots\) of speed measures on \(J\). 
\begin{definition} \label{def: LT conv}
We say that the sequence \(\m^1, \m^2, \dots\) converges in the \emph{speed measure sense} to \(\m^0\), which we denote by \(\m^n \Rightarrow \m^0\), if the following hold:
\begin{enumerate}
	\item[\textup{(a)}] \(\m^n|_{J^\circ} \to \m^0|_{J^\circ}\) vaguely, i.e. \(\int_{J^\circ} f (x) \m^n(dx) \to \int_{J^\circ} f (x) \m^0(dx)\) for all \(f \in C_c (J^\circ)\).
	\item[\textup{(b)}] If \(l \in J\), then \(\int f (x) \m^n(dx) \to \int f(x) \m^0(dx)\) for all \(0 \leq f \in C (J)\) such that \(f (l) > 0\) and \(f = 0\) off \([l, y)\) for some \(y \in J^\circ\).
	\item[\textup{(c)}] If \(r \in J\), then \(\int f (x) \m^n(dx) \to \int f(x) \m^0(dx)\) for all \(0 \leq f \in C(J)\) such that \(f (r) > 0\) and \(f = 0\) off \((y, r]\) for some \(y \in J^\circ\).
\end{enumerate}
\end{definition}
\begin{remark}
	If \(\m^0, \m^1, \m^2, \dots\) are locally finite, then \(\m^n \Rightarrow \m^0\) if and only if \(\m^n \to \m^0\) vaguely.
\end{remark}

For each \(n \in \mathbb{Z}_+\) let \((x \mapsto P^n_x)\) be the regular diffusion on natural scale with speed measure \(\m^n\). Moreover, take a sequence \(x^0, x^1, x^2, \dots \in J\). 
Among other things, Stone \cite{stone} proved the following theorem.

\begin{theorem}[\cite{stone}] \label{theo: stone}
	If \(\m^n \Rightarrow \m^0\) and \(x^n \to x^0\), then \(P^n_{x^n} \to P^0_{x^0}\) weakly.
\end{theorem}

\begin{corollary} \label{coro: cont}
If \((x \mapsto P_x)\) is a regular diffusion on natural scale, then \(x \mapsto P_x\) is a continuous function from \(J\) into \(\M\), i.e. \((x \mapsto P_x) \in C(J, \M)\).
\end{corollary}

\begin{corollary} \label{coro: loc uni}
	If \(\m^n \to \m^0\), then \((x \mapsto P^n_x) \to (x \mapsto P^0_x)\) locally uniformly. 
\end{corollary}
\begin{proof}
	Thanks to Theorem \ref{theo: stone}, for every sequence \(x^0, x^1, x^2, \dots \in J\) with \(x^n \to x^0\) we have \(P^n_{x^n} \to P^0_{x^0}\) weakly. 
	In other words, the sequence \((x \mapsto P^1_x), (x \mapsto P^2_x), \dots\) converges continuously to \((x \mapsto P^0_x)\). A theorem of Carath\'eodory (\cite[Theorem on pp. 98--99]{remmert}) shows that continuous convergence is equivalent to local uniform convergence.
\end{proof}

To explain the main idea behind Theorem \ref{theo: stone}, let us prove it for the case \(J = [0, \infty)\) and \(x^n \equiv x_0\). Of course, the main steps of the proof are borrowed from \cite{stone}.
\\

\noindent
\emph{Proof for Theorem \ref{theo: stone} in case \(J = [0, \infty)\) and \(x^n \equiv x_0\).}
The key idea is to use the It\^o--McKean construction of a regular diffusion on natural scale as a time change of Brownian motion. Let \(\B\) be a Brownian motion starting in \(x_0\) which is reflected at the origin, denote its local time process by \(\{L(t, y) \colon t, y\in \mathbb{R}_+\}\) and set 
\[
T^n_t \triangleq \int L (t, y) \m^n(dy), \quad t \in \mathbb{R}_+, \ n \in \mathbb{Z}_+.
\]
	Furthermore, set
	\[
	S^n_t \triangleq \inf (s \geq 0 \colon T^n_{s + } > t), \quad t \in \mathbb{R}_+.
	\]
		The following discussion should be read up to a null set which depends on standard properties of the local time process (see e.g. \cite[Section 2.8]{freedman}).
	As explained in the discussion below \cite[Definition 16.55]{breiman1968probability}, \(t \mapsto S^n_t\) is finite, continuous and increasing. 
	If \(t < \gamma_0 (\B)\), then \(x \mapsto L (t, x)\) is continuous and supported on a compact subsets of \(J^\circ\) such that part~(a) of Definition \ref{def: LT conv} yields that \(T^n_t \to T^0_t\).
	By the Corollary on p. 640 in \cite{stone} and the representation of the local time process for reflected Brownian motion as discussed above \cite[Definition 16.55]{breiman1968probability},
	if \(t > \gamma_0(\B)\) then the map \(x \mapsto L (t, x)\) has the same properties as \(f\) in part (b) of Definition \ref{def: LT conv}, which then yields that \(T^n_t \to T^0_t\).
	In summary, \(\m^n \Rightarrow \m^0\) implies that \(T^n_t \to T^0_t\) for all \(t \not = \gamma_0 (\B)\). 
	We conclude that \(S^n_t \to S^0_t\) for all \(t \in \mathbb{R}_+\), see \cite[Lemma 1.1.1]{HaanFre}. As \(S^0, S^1, S^2, \dots\) are increasing and continuous, \cite[Theorem VI.2.15]{JS} further yields that \(S^n \to S^0\) uniformly on compact time intervals. Hence, we also have that a.s. \(\B_{S^n} \to \B_{S^0}\) uniformly on compact time intervals. Since \(\B_{S^n}\) has law \(P^n_{x_0}\) by \cite[Theorem 16.56]{breiman1968probability}, we get that \(P^n_{x_0} \to P^0_{x_0}\) weakly.
\qed\\

The following theorem is our main result. It can be viewed as a converse to Theorem~\ref{theo: stone}. Its proof is given in Section \ref{sec: pf} below.
\begin{theorem}\label{theo: new main1}
	If \(P^n_x \to P^0_x\) weakly for all \(x \in J\), then \(\m^n \Rightarrow \m^0\). 
\end{theorem}
Combining Corollary \ref{coro: loc uni} and Theorem \ref{theo: new main1} gives us the following:
\begin{corollary} \label{coro: main1}
	The following are equivalent:
	\begin{enumerate}
		\item[\textup{(i)}] \(\m^n \Rightarrow \m^0\).
				\item[\textup{(ii)}] \((x \mapsto P^n_x) \to (x \mapsto P^0_x)\) locally uniformly.
						\item[\textup{(iii)}] \(P^n_x \to P^0_x\) weakly for all \(x \in J\).
	\end{enumerate}
\end{corollary}

It is interesting to note that on the set of regular diffusions on natural scale the sequential topologies of pointwise and local uniform convergence coincide.\\

\noindent
\emph{Comments on related literature.} 
For the real-valued case, i.e. \(J = \mathbb{R}\), Theorem \ref{theo: new main1} was proved in \cite{brooks}. Theorem \ref{theo: new main1} seems to be new in its generality. The result provides the complete picture for general state spaces and arbitrary boundary behavior.
Further, our method of proof is new and quite different to those from \cite{brooks}.  
As the method from \cite{brooks} heavily relies on the concept of vague convergence, it seems not to work for Theorem \ref{theo: new main1}, so that new ideas are necessary.
We now outline the main steps of the argument from \cite{brooks} and compare it to ours. In the following we take \(J = \mathbb{R}\). Note that \(\m^0, \m^1, \m^2, \dots\) are locally finite and that \(\m^n \Rightarrow \m^0\) if and only if \(\m^n \to \m^0\) vaguely.
The first step is to prove that the sequence \(\m^1, \m^2, \dots\) is uniformly bounded on compact subsets of the reals, which shows that \(\{\m^n \colon n \in \mathbb{N}\}\) is vaguely relatively compact (\cite[Proposition~3.16]{resnick}). In \cite{brooks} this is done by a contradiction argument.
To conclude \(\m^n \to \m^0\) vaguely it suffices to show that any vague accumulation point \(\q\) of \(\m^1, \m^2, \dots\) coincides with \(\m^0\). Assume that \(\m^n \to \q\) vaguely. 
First, \(\q\) is shown to be positive on any compact subset of the reals, which means it is a speed measure. In \cite{brooks} this is again done by a contradiction argument.
Second, let \(\B\) be a Brownian motion started at \(x_0\), denote its local time process by \(\{L(t, y) \colon t \in \mathbb{R}_+, y \in \mathbb{R}\}\) and let \(S\) be the inverse of 
\(t \mapsto \int L (t, y) \q(dy)\).
By the argument outlined in the proof of Theorem \ref{theo: stone}, \(P^n_{x_0}\) converges weakly to the law of \(\B_S\) and consequently, by the uniqueness of the limit, \(\B_S\) has law \(P^0_{x_0}\).
Now, \(\q = \m^0\) follows from the definition of the speed measure via the canonical form (i.e. as time change of Brownian motion, see \cite[Theorem~33.9]{Kallenberg}). It follows that \(\m^n \to \m^0\) vaguely.

From a technical point of view this proof heavily relies on the It\^o--McKean construction of a diffusion as a time change of Brownian motion, properties of the Brownian local time and the characterization of the speed measure via the canonical form. Our proof for Theorem~\ref{theo: new main1} uses non of these tools. Instead, we use the definition of the speed measure from the monographs \cite{breiman1968probability,RY}, a uniform second moment bound for exit times and the continuous mapping theorem.
\\

\noindent
\emph{A topological point of view.}
In the remainder of this section we look at the relation of diffusions and their speed measures from a topological point of view. 
Let \(\mathcal{S}\) be the set of all locally finite speed measures and let \(\mathcal{D}\) be the set of all completely regular diffusions on natural scale. 
Our goal is to establish a homeomorphic relation between \(\mathcal{S}\) and \(\mathcal{D}\). 
We endow \(\mathcal{S}\) with the vague topology, which turns it into a metrizible space. 
Thanks to Corollary~\ref{coro: cont}, we can treat \(\mathcal{D}\) as a subspace of \(C(J, \M)\) endowed with the local uniform topology, which renders it into a metrizible space.
\begin{remark} \label{rem: product topology}
	It would also be natural to consider regular diffusions as elements of the product space \(\M^J\). However, \(\M^J\) is not first countable (\cite[Theorem 7.1.7]{singh}) and hence we cannot a priori\footnote{The space of continuous functions \([0, 1] \to [0, 1]\) is \emph{not} sequential when endowed with the product topology, see \cite[Beispiel on p. 102]{janich}.} check continuity via sequential continuity. The space \(C(J, \M)\) on the other hand is metrizible and therefore also sequential. 
\end{remark}

Finally, let \(\Phi \colon \mathcal{D} \to \mathcal{S}\) be the function which maps a completely regular diffusion on natural scale to its speed measure. 
Corollary \ref{coro: main1} gives us the following result, which we call a theorem rather than a corollary, since we think it deserves this name.
\begin{theorem} \label{theo: main homeo}
The map	\(\Phi\) is a homeomorphism, i.e. \(\Phi\) is a continuous bijection with continuous inverse \(\Phi^{-1}\). 
\end{theorem}

	Related to Remark \ref{rem: product topology}, Corollary \ref{coro: main1} also shows the following: 
	\begin{corollary}
	In case \(\mathcal{D}\) is seen as a subspace of \(\M^J\) endowed with the product weak topology, then \(\Phi\) is a sequential homeomorphism, i.e. a sequentially continuous bijection with sequentially continuous inverse.\footnote{Of course, the inverse \(\Phi^{-1}\) is even continuous, as \(\mathcal{S}\) is sequential.} 
\end{corollary}

In the remainder of this section we apply Theorem \ref{theo: main homeo} to study properties of certain subsets of \(\mathcal{D}\). More precisely, we consider
the set of completely regular diffusions with the Feller--Dynkin property and the set of It\^o diffusions with open state space. \\

\noindent
\emph{On the set of diffusions with the Feller--Dynkin property.}
	Let \(C_0(J)\) be the set of all continuous functions \(J \to \mathbb{R}\) which are vanishing at infinity.
	We say that a diffusion \((x \mapsto P_x)\) has the \emph{Feller--Dynkin property} if \((x \mapsto E_x [f (\X_t)]) \in C_0(J)\) for all \(f \in C_0(J)\) and \(t > 0\). Let \(\mathcal{O}\) be the set of all completely regular diffusions with the Feller--Dynkin property.
	\begin{corollary} \label{prop: FD set}
	If \(J\) is bounded, then \(\mathcal{O} = \mathcal{D}\) and, in particular, \(\mathcal{O}\) is clopen in \(\mathcal{D}\). Conversely, if \(J\) is unbounded, then \(\mathcal{O}\) is a dense Borel subset of \(\mathcal{D}\) and it is neither closed nor open in \(\mathcal{D}\).
	\end{corollary}
	\begin{proof}
	According to \cite[Theorem 1.1]{criens21}, a regular diffusion with speed measure \(\m\) has the Feller--Dynkin property if and only if any infinite boundary point of \(J\) is natural,~i.e.
	\begin{align*}
	\begin{cases} \displaystyle 	\int^{\infty} x \m (dx) = \infty,&\text{ if \(\infty\) is a boundary point},\vspace{0.2cm}\\
\displaystyle \	\int_{-\infty} |x| \m (dx) = \infty,&\text{ if \(- \infty\) is a boundary point}.
	\end{cases}
	\end{align*}
	Now, if \(J\) is bounded it is clear that \(\mathcal{O} = \mathcal{D}\). 
	Suppose that \(J\) is unbounded. It is not hard to see that \(\Phi (\mathcal{O})\) is Borel but neither closed nor open.\footnote{For the latter, note e.g. that \( \frac{e^{|x| / n} dx}{|x|^3 \vee 1} \to \frac{dx}{|x|^3 \vee 1}\) and that \(e^{- |x| / n} dx \to dx\).}
	Hence, by Theorem \ref{theo: main homeo}, the same is true for \(\mathcal{O}\). Finally, the claim that \(\mathcal{O}\) is dense in \(\mathcal{D}\) follows from Theorem \ref{theo: main homeo} and the fact that any locally finite measure can be approximated in the vague topology by a sequence of discrete measures (\cite[Theorem 30.4]{Bauer}). To be more precise, let \(\m^0\in \mathcal{S}, \m \in \mathcal{O}\) and let \(\n^1, \n^2, \dots\) be a sequence of discrete measures such that \(\n^n \to \m^0\) vaguely. Then, \(\m^n (dx) \triangleq \n^n(dx) + \frac{1}{n} \m (dx)\) is the speed measure of a diffusion from \(\mathcal{O}\) and \(\m^n \to \m^0\) vaguely. This shows that \(\Phi(\mathcal{O})\) is dense in \(\mathcal{S}\) and hence, by Theorem \ref{theo: main homeo}, \(\mathcal{O}\) is dense in~\(\mathcal{D}\). The proof is complete.
	\end{proof}

\noindent
\emph{On the set of It\^o diffusions.}
 Let us assume that \(J = (l, r)\) is open.
We call a completely regular diffusion an \emph{It\^o diffusion} if its speed measure \(\m\) is absolutely continuous w.r.t. the Lebesgue measure, i.e. \(\m (dx) = f (x) dx\) for some \(f \in L_\textup{loc}^1 (J)\). Denote the set of It\^o diffusions by \(\mathcal{I}\).
\begin{corollary} \label{prop: ito diff}
	\(\mathcal{I}\) is a dense Borel subset of \(\mathcal{D}\) and it is neither closed nor open in~\(\mathcal{D}\). 
\end{corollary}
The non-closedness of the set of real-valued It\^o diffusions with drift was already observed in \cite{rosenkrantz}. Corollary \ref{prop: ito diff} provides a refined picture for the set of It\^o diffusions without drift.
\begin{proof}[Proof of Corollary \ref{prop: ito diff}]
	Let \(\mathcal{A}\) be the set of speed measures which are absolutely continuous w.r.t. the Lebesgue measure and let \(\mathcal{R}\) be the Polish space of locally finite measures on \((J, \mathcal{B}(J))\) with the vague topology.
	Furthermore, let \(S_+ (J)\) be the set of all \(f \in L^1_\textup{loc} (J)\) such that 
	\(
	\int_a^b f (x) dx > 0
	\)
	for all \(a, b \in J\) with \(a < b\) and 
	\[
		\begin{cases} \displaystyle \int^r |r - x| f(x) dx = \infty,&\text{ if \(r < \infty\)},\vspace{0.2cm}\\
	\displaystyle \ \int_l \ |l - x| f(x) dx = \infty,&\text{ if \(l > - \infty\)}.
	\end{cases}
	\]
	We endow \(L_\textup{loc}^1(J)\) with the local \(L^1\) topology which renders it into a Polish space. It is not hard to see that \(S_+ (J) \in \mathcal{B}(L_\textup{loc}^1(J))\).
		Now, consider the map \(\psi \colon S_+ (J) \to \mathcal{R}\) defined by \(\psi (f) (G) = \int_G f (x)dx\) for \(G \in \mathcal{B}(J)\). As \(\psi\) is a continuous injection from a Borel subset of a Polish space into a Polish space, \cite[Theorem~8.2.7]{cohn} yields that \(\psi (S_+ (J)) \in \mathcal{B}(\mathcal{R})\). As \(\psi (S_+ (J)) = \mathcal{A}\) by \cite[Proposition~16.43, Theorem 16.56]{breiman1968probability}, \(\mathcal{A}\) is a Borel subset of \(\mathcal{S}\).
		It is not hard to see that \(\mathcal{A}\) is neither closed nor open.\footnote{For instance, note that \(dx + ne^{-nx} \1_{\{x \geq 0\}} dx \to dx + \delta_0 (dx)\) and that  \(dx + \frac{1}{n} \delta_{x_0} (dx) \to dx\).} We conclude from Theorem~\ref{theo: main homeo} that \(\mathcal{I}\) has the same properties, i.e. it is a Borel set but neither closed nor open. Finally, let us explain that \(\mathcal{I}\) is dense in \(\mathcal{D}\). By Theorem \ref{theo: main homeo}, it suffices to show that \(\mathcal{A}\) is dense in~\(\mathcal{S}\). It is clear that any discrete measure can be approximated in the vague topology by a sequence of absolutely continuous measures.\footnote{For instance, recall that \(N(\mu, \sigma^2) \to \delta_\mu\) vaguely for \(\sigma^2 \to 0\), where \(N(\mu,\sigma^2)\) denotes the normal distribution with expectation \(\mu\) and variance \(\sigma^2\).} Thus, as the set of discrete measures is dense in \(\mathcal{R}\), it follows that \(\mathcal{A}\) is dense in \(\mathcal{S}\). The proof is complete.
\end{proof}

The remainder of this paper is devoted to the proof of Theorem \ref{theo: new main1}.

\section{Proof of Theorem \ref{theo: new main1}} \label{sec: pf}
In this section we assume that \(P^n_x \to P^0_x\) for all \(x \in J\). Our goal is to show that \(\m^n \Rightarrow \m^0\).
The proof for this is split into two parts. In the first we establish property (a) from Definition \ref{def: LT conv} and in the second we deal with the properties (b) and (c).

\subsection{Proof for convergence on the interior} \label{sec: pf a}
In this section we prove property (a) from Definition \ref{def: LT conv}.
Recall \cite[Corollary VII.3.8]{RY}:
For each \(n \in \mathbb{Z}_+\), all \(a < b\) such that \([a, b] \subset J^\circ\) and every \(f \in C_c (J)\) the following holds: 
\begin{align} \label{eq: def SM}
E_x^n \Big[ \int_0^{\gamma_a  \wedge \gamma_b  } f (\X_s) ds \Big] = \int G_{(a, b)} (x, y) f (y) \m^n (dy), \quad a < x < b, 
\end{align}
where \(G_{(a, b)}\) denotes the Green function as given in \cite[Eq. 16.35]{breiman1968probability}. 
Recall our notation \(J^\circ = (l, r)\).
\begin{lemma} \label{lem: key1}
	Suppose that \(P^n_{x_0} \to P^0_{x_0}\) weakly for some \(x_0 \in J^\circ\). Then, there exist two sets \(A \subset (l, x_0)\) and \(B \subset (x_0, r)\) with countable complements (in \((l, x_0)\) and \((x_0, r)\), respectively) such that for all \(a \in A, b \in B\) and~\(f \in C_c (J^\circ)\)
	\begin{align} \label{eq: conv to show}
	E_{x_0}^n \Big[ \int_0^{\gamma_{a}  \wedge \gamma_b } f (\X_s) ds \Big] \to E^0_{x_0} \Big[ \int_0^{\gamma_{a} \wedge \gamma_b} f (\X_s) ds \Big]
	\end{align}
	as \(n \to \infty\).
\end{lemma}

The test functions \(G_{(a, b)} (x, \cdot) f\) are sufficient to characterize vague convergence of locally finite measures on \(J^\circ\).
Thus, by virtue of the r.h.s. in \eqref{eq: def SM}, Lemma \ref{lem: key1} implies  \(\m^n |_{J^\circ} \to \m^0|_{J^\circ}\), i.e. part (a) from Definition \ref{def: LT conv}.
\qed \\

\noindent
\emph{Proof of Lemma \ref{lem: key1}.} 
For \(y \in J^\circ\) we set
\begin{align*}
\tau_y^+  &\triangleq \inf (s \geq 0 \colon \X_s \geq y), \quad \tau^-_y \triangleq \inf (s \geq 0 \colon \X_s \leq y), \\
\sigma^+_y &\triangleq \inf (s \geq 0 \colon \X_s > y), \quad \sigma^-_y \triangleq \inf (s \geq 0 \colon \X_s < y).
\end{align*}

\begin{lemma} \label{lem: semi cont}
	For any \(y \in J^\circ\) the functions \(\sigma^\pm_y\) are upper semi-continuous and the functions \(\tau^\pm_y\) are lower semi-continuous.
\end{lemma}
\begin{proof}
	The claim is implied by \cite[Exercise 2.1 on p. 75]{pinsky}. For completeness, we provide a proof for \(\sigma^+_y\) and \(\tau^+_y\). The arguments for \(\sigma^-_y\) and \(\tau^-_y\)  work the same way.
	Take \(t > 0\). We have 
	\[
	\{\sigma^+_y < t\} = \bigcup_{s \in \mathbb{Q}, s < t} \{\X_s > y\}.
	\]
	As \(\omega \mapsto \omega (s)\) is continuous for every \(s \in \mathbb{R}_+\), the set \(\{\X_s> y\}\) is open. As unions of open sets are open, \(\{\sigma^+_y < t\}\) is also open. Consequently, \(\sigma^+_y\) is upper semi-continuous.
	
	Take \(t \in \mathbb{R}_+\) and let \(d_y (x) \triangleq \inf_{z \geq y} |z - x|\) for \(x \in J\). We have 
	\[
	\{\tau^+_y \leq t\} = \Big\{\inf_{s \in \mathbb{Q}, s \leq t} d_y (\X_s) = 0\Big\}.
	\]
	For every \(s \in \mathbb{Q} \cap [0, t]\) and \(\omega, \omega' \in \Omega\) we also have
	\begin{align*}
	\inf_{r \in \mathbb{Q},r \leq t} d_y (\omega (r)) \leq d_y (\omega (s)) 
	&\leq \sup_{r \leq t} | \omega(r) - \omega'(r)| + d_y (\omega' (s)).
	\end{align*}
	Taking the infimum over \(s\) and using symmetry yields that
	\[
	\Big| \inf_{r \in \mathbb{Q},r \leq t} d_y (\omega (r)) - \inf_{r \in \mathbb{Q},r \leq t} d_y (\omega' (r))\Big| \leq \sup_{r \leq t} |\omega(r) - \omega'(r)|.
	\]
	Consequently, \(\omega \mapsto \inf_{s \in \mathbb{Q}, s \leq t} d_y (\omega(s))\) is continuous, \(\{\tau^+_y \leq t\}\) is closed and \(\tau^+_y\) is lower semi-continuous.
\end{proof}
Let \(A\) be the set of all \(a \in (l, x_0)\) such that \(\tau^-_a\) is \(P^0_{x_0}\)-a.s. continuous and let \(B\) be the set of all \(b \in (x_0, r)\) such that \(\tau^+_b\) is \(P^0_{x_0}\)-a.s. continuous.

\begin{lemma} The complements of \(A\) and \(B\) are both at most countable.\end{lemma} \label{lem: count}
\begin{proof}
	We restrict our attention to the set \(B\). The claim for \(A\) follows the same way. 
	Note that \(\sigma^+_b = \tau^+_{b+}\) for all \(b \in J^\circ\). It is well-known (\cite[Lemma 7.7 on p. 131]{EK}) that the set \(\{b> x_0 \colon P^0_{x_0}(\tau^+_b \not = \tau^+_{b+}) > 0\}\) is a most countable. 
	Consequently, by virtue of Lemma \ref{lem: semi cont}, the complement of \(B\) is at most countable. 
\end{proof}

It remains to prove \eqref{eq: conv to show}. The key step is the following lemma.
\begin{lemma} \label{lem: int1}
	\(\sup_{n \in \mathbb{N}} E_{x_0}^n \big[ \big(\tau_a^- \wedge \tau^+_b\big)^2 \big] < \infty\) for all \(a < x_0 < b\).
\end{lemma}
To prove this we first establish two preliminary results.
	\begin{lemma} \label{lem: help1}
	For all \(n \in \mathbb{Z}_+, t > 0\) and \(a \leq x, x_0 \leq b\) we have 
	\[P_x^n (\tau^-_a \wedge \tau^+_b > t) \leq P^n_{x_0} (\sigma^-_a \geq t) \vee P^n_{x_0} (\sigma^+_b\geq t).\]
\end{lemma}
\begin{proof}
	Let us first take \(a \leq x \leq x_0\). The strong Markov property yields that 
	\begin{align*}
	P_{x_0}^n (\sigma^-_a < t) &\leq P_{x_0}^n (\tau^-_a < t) 
	\\&\leq P_{x_0}^n (\tau^-_x  < \infty, \tau_a^- (\theta_{\tau^-_x}\X) \leq t)
	\\&= E_{x_0}^n \big[ \1_{\{\tau^-_x < \infty\}} P_{\X_{\tau^-_x}}^n (\tau^-_a \leq t)\big]
	\\&= P_{x_0}^n (\tau^-_x < \infty) P_x^n (\tau^-_a \leq t)
	\\&\leq P_x^n (\tau^-_a  \wedge \tau^+_b \leq t).
	\end{align*}
	This yields the claimed inequality for \(a \leq x \leq x_0\).
	
	For \(x_0 \leq x \leq b\) we get from the same computation that
	\[
	P_{x_0}^n (\sigma^+_b < t) \leq P_x^n (\tau^+_b\leq t) \leq P_x^n (\tau^-_a  \wedge \tau^+_b \leq t).
	\]
	The proof is complete.
\end{proof}	
	\begin{lemma} \label{lem: alpha less 1}
		For all \(t > 0\) and \(a < b\) we have 
	\begin{align} \label{eq: to show1}
	\alpha \triangleq \sup ( P_x^n (\tau^-_a  \wedge \tau^+_b > t) \colon n \in \mathbb{Z}_+,a \leq x \leq b) < 1.
	\end{align}
	\end{lemma}
\begin{proof}
	By Lemma \ref{lem: semi cont}, \(\sigma_{a}^\pm\) are upper semi-continuous. Hence, by \cite[Theorem~15.5]{aliprantis}, the maps
	\(
\M\ni P \mapsto P (\sigma_{a}^\pm \geq t) \in [0, 1]
	\)
	are also upper semi-continuous. 
	As we assume that \(P^n_{x_0} \to P^0_{x_0}\) weakly, the set \(\{P_{x_0}^n \colon n \in \mathbb{Z}_+\}\) is compact in \(\M\). Now, since upper semi-continuous functions attain a maximum value on a compact set (\cite[Theorem 2.43]{aliprantis}), there exists an \(N_{a}^\pm \in \mathbb{Z}_+\) such that \[\sup_{n \in \mathbb{Z}_+} P_{x_0}^n (\sigma_{a}^\pm \geq t) = P_{x_0}^{N_{a}^\pm} (\sigma_{a}^\pm \geq t).\]
	As, thanks to \cite[Theorem 1.1]{bruggeman}, regular diffusions hit points arbitrarily fast with positive probability, we have \[P_{x_0}^{N_{a}^\pm} (\sigma_{a}^\pm  \geq t) < 1\] and \eqref{eq: to show1} follows from Lemma \ref{lem: help1}.
	\end{proof}
	\begin{proof}[Proof of Lemma \ref{lem: int1}]
		We fix \(t > 0\).
	Using the Markov property, for every \(m \in \mathbb{Z}_+\) we get
	\begin{align*}
	P_{x_0}^n (\tau^-_a  \wedge \tau^+_b > mt + t ) 
	&= P_{x_0}^n (\tau^-_a  \wedge \tau^+_b> mt, \tau^-_a (\theta_{mt}\X) \wedge \tau^+_b (\theta_{mt}\X) > t )
	\\&= E_{x_0}^n \big[ \1_{\{\tau^-_a\wedge \tau^+_b > mt\}} P_{\X_{mt}}^n (\tau^-_a \wedge \tau^+_b > t)\big]
	\\&\leq P_{x_0}^n (\tau^-_a  \wedge \tau^+_b > mt) \ \alpha.
	\end{align*}
	Thus, by induction we obtain for every \(m \in \mathbb{Z}_+\) that
	\[
	P_{x_0}^n(\tau^-_a \wedge \tau^+_b > mt) \leq \alpha^m. 
	\]
	Finally, we can estimate
	\begin{align*}
	E^n_{x_0} \big[ \big(\tau^-_a \wedge \tau^+_b\big)^2 \big] 
	&= \sum_{m = 0}^\infty \int_{mt}^{(m + 1)t} 2 s P^n_{x_0}(\tau^-_a \wedge \tau^+_b > s) ds
	\leq \sum_{m = 0}^\infty (m + 1)^2 t^2 \alpha^{m}. 
	\end{align*}
	As the final term is finite by Lemma \ref{lem: alpha less 1} and independent of \(n\), the proof is complete.
\end{proof}

We are in the position to finish the proof of Lemma \ref{lem: key1}. Take \(f \in C_c (J^\circ), a \in A\) and \(b \in B\). 
First of all, as \(a < x_0 < b\), for every \(n \in \mathbb{Z}_+\) we have \(P^n_{x_0}\)-a.s. 
\[
\int_0^{\gamma_a \wedge \gamma_b} f (\X_s) ds = \int_0^{\tau^-_a\wedge \tau^+_b} f (\X_s) ds.
\]
Then, by Lemma \ref{lem: int1}, we have 
\[
\sup_{n \in \mathbb{N}}E^n_{x_0} \Big[ \Big( \int_0^{\tau^-_a \wedge \tau^+_b} f (\X_s) ds \Big)^2 \Big] \leq \|f\|^2_\infty \sup_{n \in \mathbb{N}} E^n_{x_0} \big[ \big(\tau^-_a \wedge \tau^+_b \big)^2 \big] < \infty.
\]
Thus, the family \[\Big\{P^n_{x_0} \circ \Big(\int_0^{\tau^-_a  \wedge \tau^+_b} f(\X_s) ds\Big)^{-1} \colon n \in \mathbb{N}\Big\}\] is uniformly integrable. Furthermore, by definition of the sets \(A\) and \(B\), the function 
\[
\omega \mapsto \int_0^{\tau^-_a (\omega) \wedge \tau^+_b (\omega)} f (\omega(s)) ds
\]
is \(P^0_{x_0}\)-a.s. continuous.
Consequently, the continuous mapping theorem yields \eqref{eq: conv to show}. The proof of Lemma \ref{lem: key1} is complete. 
\qed

\subsection{Proof of convergence up to the boundaries}
We now prove property (b) from Definition \ref{def: LT conv}. Assume that \(l \in J\). 
In the following we distinguish between the cases where \((x \mapsto P^0_x)\) is absorbing or reflecting\footnote{We call a (closed) boundary point \emph{reflecting} if it is instantaneously or slowly reflecting in the sense of \cite[Section 16.7]{breiman1968probability}. Further, we call a closed boundary point \emph{absorbing} if it is not reflecting. Thus, a boundary point is absorbing if it is either exit or regular--absorbing in the sense of \cite[Section 16.7]{breiman1968probability}.} at the boundary point \(l\). 
\\

\noindent
\emph{The absorbing case.}
Assume that \(l\) is an absorbing boundary point of \((x \mapsto P^0_x)\). This case is captured by the speed measure via the property \(\m^0(\{l\}) = \infty\).
For every \(b \in J^\circ\) and \(t > 0\) the set \(\{\tau^+_b > t\}\) is open by Lemma \ref{lem: semi cont} and, using the Portmanteau theorem, we~get
\[
\liminf_{n \to \infty} P^n_l (\tau^+_b > t) \geq P^0_l (\tau^+_b > t) = 1.
\]
In other words, if \(X^1, X^2, \dots\) are random variables with laws \(P^1_l \circ (\tau^+_b)^{-1}, P^2_l \circ (\tau^+_b)^{-1}, \dots\), then \(X^n \to \infty\) in probability, which yields that \(E^n_l [ \tau^+_b] \to \infty\). Now, take \(0 \leq f \in C(J)\) such that \(f (l) > 0\) and \(f = 0\) off \([l, y)\) for some \(y \in J^\circ\).
Let \(b \in (l, y)\) be such that \(f > 0\) on \([l, b]\) and choose \(y < z \in J^\circ\).
By \cite[Proposition VII.3.10]{RY}, we have 
\begin{align}\label{eq: sym GF}
E^n_l \Big[ \int_0^{\tau^+_{z}} \frac{f (\X_s) ds}{G_{[l, z)} (l, \X_s)} \Big] = \int f (x) \m^n(dx)
\end{align}
where \(G_{[l, z)}\) is the symmetrized Green function as given in \cite[Eq. 16.46]{breiman1968probability}. Now, we obtain
\[
\int f (x) \m^n(dx) \geq E^n_l \Big[ \int_0^{\tau^+_{b}} \frac{f (\X_s) ds}{G_{[l, z)} (l, \X_s)} \Big] \geq \min_{x \in [l, b]} \frac{f (x)}{G_{[l, z)} (l, x)} E^n_l \big[ \tau^+_b \big] \to \infty.
\]
This completes the proof for the absorbing case.
\\

\noindent
\emph{The reflecting case.}
Next, we assume that \(l\) is a reflecting boundary point of \((x \mapsto P^0_x)\). 
Take \(z \in J^\circ\) and \(t > 0\). Using Lemma \ref{lem: semi cont}, the Portmanteau theorem and \cite[Theorem~1.1]{bruggeman}, we get
\[
\liminf_{n \to \infty} P^n_l (\sigma^+_z < t) \geq P^0_l (\sigma^+_z < t) > 0.
\]
Recall the Urysohn property: A sequence in a metrizible space converge to a limit \(L\) if and only if any of its subsequences contains a further subsequence which converges to \(L\). Thus, to prove \(\int f (x) \m^n(dx) \to \int f (x) \m^0(dx)\) we have to prove that any subsequence of \(\int f (x) \m^1(dx),\) \(\int f(x) \m^2(dx), \dots\) contains a further subsequence which converges to \(\int f(x) \m^0(dx)\). Let \(k(1), k(2), \dots\) be an arbitrary subsequence of \(1, 2, \dots\). Then, 
\[
\liminf_{n \to \infty} P^{k(n)}_l (\sigma^+_z < t) \geq \liminf_{n \to \infty} P^n_l (\sigma^+_z < t) > 0.
\]
Thus, there exists a subsequence \(m(1), m(2), \dots\) of \(k(1), k(2), \dots\) such that \[P^{m(n)}_l (\sigma^+_z < t) > 0, \quad \forall n \in \mathbb{N}.\]  Let us take this subsequence and prove part (b) of the definition of \(\m^{m(n)} \Rightarrow \m^0\). Once we have done this, we can conclude part (b) of \(\m^n \Rightarrow \m^0\).
To simplify our notation, we assume that \(P^n_l (\sigma^+_z < t) > 0\) for every \(n \in \mathbb{N}\). Consequently, for each \(n \in \mathbb{Z}_+\) the point \(l\) is a reflecting boundary of \((x \mapsto P^n_x)\) and, recalling again \cite[Theorem 1.1]{bruggeman}, we have 
\begin{align}\label{eq: hypo last pf} 
P^n_l (\sigma^+_c < t) > 0 \text{ for all } n \in \mathbb{Z}_+ \text{ and } c \in J^\circ.
\end{align}

We claim that for all but countably many \(b \in (l, r)\) and any \(f \in C_c (J)\)
\begin{align} \label{eq: conv to show ii}
E^n_l \Big[ \int_0^{\tau^+_b } f (\X_s) ds \Big] \to E^0_l \Big[ \int_0^{\tau^+_b} f (\X_s) ds \Big].
\end{align}
Recalling \eqref{eq: sym GF}, this yields property (b) of \(\m^n \Rightarrow \m^0\).
As in Section \ref{sec: pf a}, we can take \(b \in (l, r)\) such that the map
\[
\omega \mapsto \int_0^{\tau^+_b (\omega)} f (\omega(s)) ds
\]
is \(P^0_l\)-a.s. continuous for all \(f \in C_c(J)\). Thanks to the continuous mapping theorem, it now suffices to prove that the family \[\Big\{P^n_l \circ \Big(\int_0^{\tau^+_b} f (\X_s) ds\Big)^{-1} \colon n \in \mathbb{N}\Big\}\] is uniformly integrable, which follows for instance from
\begin{align}\label{eq: second moment bound}
\sup_{n \in \mathbb{N}} E^n_l \big[ ( \tau^+_b )^2 \big] < \infty.
\end{align}
Set
\[
\beta \triangleq \sup ( P_x^n (\tau^+_b > t) \colon n \in \mathbb{Z}_+, l \leq x \leq b).
\]
Provided \(\beta < 1\), we get as in the proof of Lemma \ref{lem: int1} that 
\[
E_l^n \big[ (\tau^+_b)^2 \big] \leq \sum_{m = 0}^\infty (m + 1)^2 t^2 \beta^m < \infty, 
\]
which implies \eqref{eq: second moment bound}. Thus, \eqref{eq: conv to show ii} is proved once we show that \(\beta < 1\). As in the proof of Lemmata~\ref{lem: help1}~and~\ref{lem: alpha less 1}, we obtain that
\[
\beta \leq P_{l}^{N^+_b} (\sigma^+_b \geq t), 
\]
for some \(N^+_b \in \mathbb{Z}_+\). Thanks to \eqref{eq: hypo last pf}, we get \[P_{l}^{N^+_b} (\sigma^+_b \geq t) < 1\] and \eqref{eq: conv to show ii} holds. This completes the proof for the reflecting case.
\\

\noindent
Finally, let us comment on the remaining property (c) from Definition \ref{def: LT conv}. Its proof is similar to those for the property (b). Hence, we omit the details and leave them to the reader. We conclude the proof of Theorem \ref{theo: new main1}.
\qed

\bibliographystyle{plain}

\end{document}